\documentclass[12pt]{amsart}
\usepackage{latexsym,amssymb,amscd}

\def\NZQ{\Bbb}               
\def\NN{{\NZQ N}}

\def\ZZ{{\NZQ Z}}

\def\frk{\frak}               

\def\pp{{\frk p}}

\def\B'c{{\mathcal{B'}}}
\def\U'c{{\mathcal{U'}}}

\def\opn#1#2{\def#1{\operatorname{#2}}} 
\opn\chara{char}
\opn\length{\ell}
\opn\cd{cd}
\opn\projdim{proj\,dim}
\opn\injdim{inj\,dim}
\opn\ini{in}
\opn\rank{rank}
\opn\depth{depth}
\opn\height{ht}
\opn\bigheight{bight}
\opn\embdim{emb\,dim}
\opn\codim{codim}

\opn\Tr{Tr}
\opn\bigrank{big\,rank}
\opn\superheight{superheight}\opn\lcm{lcm}
\opn\trdeg{tr\,deg}%
\opn\reg{reg}
\opn\lreg{lreg}
\opn\set{set}
\opn\supp{Supp}
\opn\shad{Shad}
\opn\indeg{indeg}
\opn\lex{lex}
\opn\div{div}
\opn\Div{Div}
\opn\cl{cl}
\opn\Cl{Cl}
\opn\Spec{Spec}
\opn\Supp{Supp}
\opn\supp{supp}
\opn\Sing{Sing}
\opn\Ass{Ass}
\opn\Ann{Ann}
\opn\Rad{Rad}
\opn\Soc{Soc}
\opn\Ker{Ker}
\opn\Coker{Coker}
\opn\Im{Im}
\opn\Hom{Hom}
\opn\Tor{Tor}
\opn\Ext{Ext}
\opn\End{End}
\opn\Aut{Aut}
\opn\id{id}

\opn\nat{nat}
\opn\GL{GL}
\opn\SL{SL}
\opn\mod{mod}
\opn\ord{ord}
\opn\ara{ara}
\opn\aff{aff}
\opn\con{conv}
\opn\relint{relint}
\opn\st{st}
\opn\lk{lk}
\opn\cn{cn}
\opn\core{core}
\opn\vol{vol}
\opn\gr{gr}

\def\pot#1#2{#1[\kern-0.28ex[#2]\kern-0.28ex]}

\opn\dirlim{\underrightarrow{\lim}}
\opn\invlim{\underleftarrow{\lim}}
\def\pnt{{\raise0.5mm\hbox{\large\bf.}}}

\def\Implies{\ifmmode\Longrightarrow \else
     \unskip${}\Longrightarrow{}$\ignorespaces\fi}
\def\implies{\ifmmode\Rightarrow \else
     \unskip${}\Rightarrow{}$\ignorespaces\fi}
\def\iff{\ifmmode\Longleftrightarrow \else
     \unskip${}\Longleftrightarrow{}$\ignorespaces\fi}

\let\:=\colon
\newtheorem{Theorem}{Theorem}[section]
\newtheorem{Lemma}[Theorem]{Lemma}
\newtheorem{Corollary}[Theorem]{Corollary}
\newtheorem{Proposition}[Theorem]{Proposition}

\let\epsilon=\varepsilon
\let\phi=\varphi
\let\kappa=\varkappa
\title{Arithmetical rank of lexsegment edge ideals }
\author{Viviana Ene \and Oana Olteanu \and Naoki Terai}
\thanks{\footnotesize  The third author was supported by Kakenhi20540047
}

\address{Faculty of Mathematics and Computer Science, Ovidius University, Bd.\ Mamaia 124,
 900527 Constanta, Romania,} \email{vivian@univ-ovidius.ro} 
\address{Faculty of Mathematics and Computer Science, Ovidius University, Bd.\ Mamaia 124,
 900527 Constanta, Romania,} \email{olteanuoanastefania@gmail.com} 
\address{Department of Mathematics, Faculty of Culture and Education, Saga University, Saga 840-88502, Japan,}
 \email{terai@cc.saga-u.ac.jp}

\begin{document}

\maketitle

\begin{abstract} Let $I\subset S=K[x_1,\ldots,x_n]$ be a lexsegment edge ideal or the Alexander dual of such an ideal. 
 In both cases it turns out that 
the arithmetical rank of $I$ is equal to the projective dimension of $S/I.$ \\



\end{abstract}

\section*{Introduction}

Let $S=K[x_1,\ldots,x_n]$ be the polynomial ring in $n$ variables over a field $K$. Let $I\subset S$ be a homogeneous ideal and 
$\sqrt{I}$ its radical. The \textit{arithmetical rank} of $I$ is defined as
\[
\ara(I)=\min\{r\in \NN\colon \text{ there exist }a_1,\ldots,a_r\in I \text{ such that }\sqrt{I}=\sqrt{(a_1,\ldots,a_r)}\}.
\]

Geometrically, $\ara(I)$ is the smallest number of hypersurfaces whose intersection is set-theoretically equal to the algebraic set defined 
by $I$, if $K$ is algebraically closed.

For a squarefree monomial ideal $I\subset S$ the following upper bound of $\ara(I)$ is known \cite{Gr}. Namely,
\[
\ara(I)\leq n-\indeg(I)+1,
\]
where $\indeg(I)$ is the \textit{initial degree} of $I,$ that is, $\indeg(I)=\min\{q\colon I_q\neq 0\}.$

Let $\cd(I)=\max\{i\in \ZZ\colon H_I^i(S)\neq 0\},$ where $H_I^i(S)$ denotes the $i$-th local cohomology module of $S$ with 
support at $V(I).$ The number $\cd(I)$ is called the \textit{cohomological dimension} of $I.$ By expressing the local cohomology modules in terms of Cech complex, one can see that $\ara(I)$ is  bounded below
by $\cd(I).$ For a squarefree monomial ideal $I$ of $S$, it is known that $\cd(I)=\projdim_S(S/I)$ (see \cite{L}). From these inequalities we 
get 
\begin{equation}\label{equstar}
\projdim_S(S/I)=\cd (I)\leq \ara(I)\leq n-\indeg(I)+1,
\end{equation} 
for any squarefree monomial ideal $I\subset S.$

There are many instances when the equality $\projdim_S(S/I)=\ara(I)$ holds. We refer the reader to \cite{B2}, \cite{BT}, \cite{BT1}, 
\cite{K}, \cite{KTY}, \cite{KTY1}, \cite{Ku}, \cite{M}, \cite{SV}, \cite{ScV} for classes of ideals $I\subset S$ whose arithmetical rank is equal to the projective dimension of $S/I$.

We show in Section \ref{aralexideals} that the  equality also holds for a lexsegment edge ideal. By a \textit{lexsegment edge ideal} we 
mean a squarefree monomial ideal generated in degree two by a lexsegment set, that is, a set of the form 
$$L(u,v)=\{w\colon w \text{ is a squarefree monomial of degree }2, \ u\geq_{\lex}w \geq_{\lex} v\},$$ where $u\geq_{\lex} v$ are two 
squarefree monomials of degree $2$ in $S.$ In order to prove the equality $\ara(I)=\projdim_S(S/I)$ for any lexsegment edge ideal we 
need first to compute  some invariants of these classes of ideals. We make these computations in Section \ref{invariants}. Having the formulas for dimension and depth, we recover the characterization of the Cohen-Macaulay lexsegment edge ideals from \cite{BS2}. Moreover, since by Theorem \ref{aralei}, $\ara(I)=\projdim_S(S/I)$ for any lexsegment edge ideal $I$, it turns out that any Cohen-Macaulay lexsegment edge ideal is a set-theoretically complete intersection too. In the last 
section we show 
that, given a lexsegment edge ideal $I,$ we have  the equality 
$\ara(I^{\ast})=\projdim_S(S/I^{\ast})$ for its Alexander dual $I^{\ast}$ as well.

\section{Invariants of lexsegment edge ideals}
\label{invariants}

Given a lexsegment edge ideal $I,$ we are going to determine $\dim(S/I),$ $\depth(S/I),$ and $\reg(I).$
Let $u=x_1x_i, v=x_jx_k, j<k,$ be two squarefree monomials of degree $2$ such that $u\geq_{\lex} v$ and $I=(L(u,v))$ the lexsegment edge 
ideal generated by the set $L(u,v)$. 

We always assume that the set $L(u,v)$ contains at least two elements, that is, $u>_{\lex} v$. 

Moreover, for our study, we may consider that $x_1|u$. Indeed, if $u=x_lx_q$ for some $l\geq 2$, then $x_1,\ldots,x_{l-1}$ is a regular sequence on $S/I$, and we may reduce to the computation of all the invariants in the ring of polynomials in the variables $x_l,\ldots,x_n$. 

We first recall the well-known fact that if $u=x_1x_2$ and $v=x_{n-1}x_n,$ that is $I$ is equal to the ideal $I_{n,2}$ generated 
by all the squarefree monomials of degree two in $n$ variables, then we have $\dim(S/I)=\depth(S/I)=1$ and $I$ has a linear resolution, 
that is, $\reg(I)=2.$ Therefore, further on, we always consider that $I\neq I_{n,2}.$ Moreover, we notice that one may also assume that 
$j\geq 2.$ Indeed, if $j=1,$ that is, $v=x_1x_k$ for some $k\geq i,$ then all the invariants can be easily computed.  

We begin our study with the computation of $\dim(S/I).$ 

If $I$ is an initial lexsegment edge ideal, that is, $I$ is generated by a lexsegment set $L^i(v)=L(u,v)$ where
 $u=x_1x_2,$ then, by \cite[Proposition 1.1]{AHH}, we get $\dim(S/I)=n-j.$ In 
the next lemma we compute $\dim(S/I)$ for a final lexsegment edge ideal, that is, generated by a lexsegment set $L^f(u)=L(u,v)$ 
where $v=x_{n-1}x_n.$

\begin{Lemma}
Let $I=(L^f(u)),$ where $u=x_1x_i, i\geq 3.$ Then $\dim(S/I)=2.$
\end{Lemma}

\begin{proof}
Let $\pp$ be a minimal prime ideal of $I.$ Then $\pp$ contains the ideal generated by all  the squarefree monomials of degree $2$ in the 
variables $x_2,\ldots,x_n$ whose height is $n-2$, hence $\height (\pp)\geq n-2,$ which implies that $\height(I)\geq n-2.$ On the other
 hand, since $(x_3,\ldots,x_n)\supset I,$ we get $\height(I)\leq n-2.$ Consequently, $\height(I)=n-2$ and $\dim(S/I)=2.$
\end{proof}

\begin{Proposition}
Let $I=(L(u,v))$ be a lexsegment edge ideal which is neither initial nor final and is determined by $u=x_1x_i$ and $v=x_jx_k.$ Then 
$\dim(S/I)=n-j.$
\end{Proposition}

\begin{proof}
We clearly have $i\geq 3$ and we may assume that $2\leq j\leq n-2.$ It is obviously that $(x_1,\ldots,x_j)\supset I,$ hence 
$\height(I)\leq j,$ 
and, therefore, $\dim(S/I)\geq n-j.$ We show that the other inequality holds as well. This is obvious if $j=2$ since $(x_1,x_2)$ is a 
minimal prime ideal of $I.$ Thus we take $j\geq 3.$

Let us consider $\pp$ a prime ideal which contains $I.$ We distinguish two cases.

Case (a). Let $x_1\in\pp,$ that is, $\pp=(x_1)+\pp^{\prime}$ where $\pp^{\prime}$ is generated by a subset of the set 
$\{x_2,\ldots,x_n\}.$ As $I\subset \pp,$ it follows that $\pp^{\prime}$ contains the initial lexsegment defined by $v$ in the ring 
$K[x_2,\ldots,x_n].$ Therefore, $\height(\pp^{\prime})\geq j-1,$ by \cite[Proposition 1.1.]{AHH}, whence $\height(\pp)\geq j.$

Case (b). Let $x_1\not\in \pp.$ Then $x_i,\ldots,x_n\in\pp,$ that is, $\pp$ has the form $\pp=(x_i,\ldots,x_n)+\pp^{\prime},$ 
where $\pp^{\prime}$ is generated by a subset of $\{x_2,\ldots,x_{i-1}\}.$ We need to consider the following subcases.

Subcase (b1). $x_{i-2}x_{i-1}\geq_{\lex}v.$ Then the ideal generated by all the squarefree monomials of degree $2$ in the variables 
$x_2,\ldots,x_{i-1}$ is contained in $\pp^{\prime}$ which implies that $\height(\pp^{\prime})\geq i-3,$ thus 
$\height(\pp)\geq n-2\geq j.$

Subcase (b2). Let $x_{i-2}x_{i-1}<_{\lex} v.$ Then $\pp^{\prime}$ contains the initial ideal $(L^i(v))\subset K[x_2,\ldots,x_{i-1}].$
It follows that $\height(\pp^{\prime})\geq j-1$ and, therefore, $\height(\pp)\geq n-i+j\geq j.$

Consequently, in all cases, we get $\height(\pp)\geq j$ for any prime ideal $\pp\supset I$ which implies the inequality 
$\dim(S/I)\leq n-j.$ 
\end{proof}

In the second part of this section we compute the depth of $S/I$ for an arbitrary lexsegment edge ideal $I.$

\begin{Proposition}\label{depth1}
Let $u=x_1x_i, v=x_jx_k$ with $j\geq 2,$ and $I=(L(u,v)).$ Then $\depth(S/I)=1$ if and only if $x_{i-1}x_n\geq_{\lex} v.$
\end{Proposition}

\begin{proof}
Let $\Delta$ be the simplicial complex on the vertex set $[n]$ whose Stanley-Reisner ideal is $I.$ It is known that $\depth(S/I)=1$ 
if and only if $\Delta$ is disconnected, which, in turn, is equivalent to the fact that the skeleton 
$\Delta^{(1)}=\{F\in \Delta\colon \dim F\leq 1\}$ of $\Delta$ is disconnected. 

In the first place we consider $\Delta^{(1)}$ disconnected. Let $V_1,V_2\neq \emptyset,$ $V_1\cup V_2=[n],$ $V_1\cap V_2=\emptyset,$ 
and such that no face of $\Delta^{(1)}$ has vertices in both $V_1$ and $V_2.$ One may assume that $1\in V_1.$ Then, since 
$\{1,2\},\ldots,\{1,i-1\}\in \Delta^{(1)},$ we must have $2,\ldots,i-1\in V_1.$ Let us assume that $v>_{\lex}x_{i-1}x_n.$ Then 
$\{\ell, n\}\in \Delta^{(1)}$ for all $\ell\geq i-1$ which implies that $i,\ldots,n\in V_1$ as well. This leads to $V_1=[n]$ which is a 
contradiction to our hypothesis.

For the converse, let $x_{i-1}x_n\geq_{\lex} v.$ We claim that $\Delta^{(1)}$ is disconnected. Indeed, one may choose 
$V_1=\{1,\ldots,i-1\}$ and $V_2=\{i,\ldots,n\}$ and observe that for any $1\leq r\leq i-1$ and $i\leq s\leq n$ we have $x_rx_s\in I$,
hence $\{r,s\}\not\in \Delta^{(1)}.$
\end{proof}

\begin{Corollary}\label{pd1}
Let $u$ and $v$ as in the above proposition. Then $\projdim_S(S/I)=n-1$ if and only if $x_{i-1}x_n\geq_{\lex}v.$
\end{Corollary}

Next we compute the depth of $S/I$ in the case when $v=x_jx_k$ with $j\geq 2$ and $v>_{\lex}x_{i-1}x_n.$ In the next lemma we 
investigate the case  $j\geq 3.$

\begin{Lemma}\label{depth2}
Let $I=(L(u,v))$ where $u=x_1x_i, v=x_jx_k$, $j\geq 3,$ and $v>_{\lex}x_{i-1}x_n.$ Then $\depth(S/I)=2.$
\end{Lemma}

\begin{proof}
By the hypothesis on $v$ we have $\depth(S/I)\geq 2.$ Let $\Delta$ be the simplicial complex on $[n]$ such that $I=I_{\Delta}.$  We 
claim that $\{1,2\}$ is a facet of $\Delta.$ Indeed, if $3\leq 
p\leq n,$ then $\{1,2,p\}\not\in \Delta$ since $x_2x_p\in I_{\Delta}.$ Thus $(x_3,\ldots,x_n)$ is a minimal prime of $I$ and so 
$\depth (S/I)\leq 2.$
\end{proof}

It remains to consider the case $v=x_2x_k$ for some $k\geq 3.$

\begin{Lemma}\label{depth3}
Let $u=x_1x_i, v=x_2x_k>_{\lex}x_{i-1}x_n$ and $I=(L(u,v)).$ Then
\[
\depth(S/I)=\left\{
\begin{array}{ll}
	2, & \text{ if }\ k\geq i,\\
	i+1-k, & \text{ if }\ i>k.
\end{array}\right.
\]
\end{Lemma}

\begin{proof}
Let us first consider $k\geq i.$ One may easily see that $I$ has the following primary decomposition
\[
I=(x_1,x_2)\cap(x_1,x_3,\ldots,x_k)\cap(x_2,x_i,\ldots,x_n)\cap(x_3,\ldots,x_n).
\]
Hence $\depth(S/I)\leq 2,$ which is enough by Proposition \ref{depth1}.

For $i>k$ one checks that the minimal monomial generators of $I,$ let us say, $m_1,\ldots,m_r,$ satisfy the following condition: 
for any $1\leq i\leq r,$ there exists $1\leq j\leq n$ such that $x_j| m_i$ and $x_j\not|m_{\ell}$ for all $\ell\neq i.$ This implies that the Taylor resolution of $S/I$ is minimal and, therefore, $\projdim_S(S/I)$ is equal to the number of the minimal monomial generators of $I$, that is, $\projdim_S(S/I)=n+k-i-1.$ Consequently, $\depth(S/I)=i+1-k.$
\end{proof}

Based on the above formulas for dimension and depth we can easily recover the characterization of the Cohen-Macaulay lexsegment edge ideals given in \cite{BS2}.

\begin{Corollary}
Let $I=(L(u,v))$ be a lexsegment edge ideal with $x_1|u$ and $u \ne v$. Then $I$ is Cohen-Macaulay if and only if one of the following conditions holds:
\begin{itemize}
	\item [(i)] $I=I_{n,2}$.
	\item [(ii)] $u=x_1x_n$ and $v\in\{x_2x_3,x_{n-2}x_{n-1},x_{n-2}x_n\}$ for $n\geq 4$.
	\item [(iii)] $u=x_1x_{n-1}$, $v=x_{n-2}x_{n-1}$ for $n\geq 3$.
\end{itemize}
\end{Corollary}

In the last part of this section we compute the regularity of a lexsegment edge ideal.

We first notice that if $I$ is an initial or final lexsegment edge ideal, then $\reg(I)=2$ since $I$ has a linear resolution.
Therefore we may consider that $u\neq x_1x_2$, that is, $i\geq 3,$ and $v\neq x_{n-1}x_n,$ in other words, $2\leq j\leq n-2.$

\begin{Lemma}
Let $I=L(u,v)$ be a lexsegment edge ideal. Then $\reg(I)\in \{2,3\}.$
\end{Lemma}

\begin{proof}
The ideal $I$ can be decomposed as $I=J+J^{\prime}$ where $J$ is generated by the lexsegment $L(u,x_1x_n)$ and $J^{\prime}$ by 
$L(x_2x_3,v).$ Both ideals $J$ and $J^{\prime}$ have a linear resolution, hence $\reg(J)=\reg(J^{\prime})=2$. By \cite{KM} (see also \cite{H} and \cite{T}), it follows that $\reg(I)\leq \reg(J)+\reg(J^{\prime})-1=3.$
\end{proof}

This easy lemma shows that we have to distinguish only between two possible values of the regularity of $I.$

In the first place we recall the characterization of the squarefree lexsegment ideals of arbitrary degree which have a linear resolution
(see \cite{BS} or \cite{BEOS}). The characterization depends on whether or not the lexsegment is complete. For the next two 
results we recall the following well-known notation. If $w\in S$ is a monomial we denote $\max(w)=\max\{j\colon x_j|w\}$ and 
$\min(w)=\min\{j\colon x_j|w\}.$

\begin{Theorem}[\cite{BS},\cite{BEOS}]\label{caractcomp}
Let $u=x_1x_{i_2}\cdots x_{i_d}$ and $v=x_{j_1}\cdots x_{j_d}$ be two squarefree monomials of degree $d\geq 2$ and $I=(L(u,v))$ the 
squarefree lexsegment ideal generated by the lexsegment set $L(u,v).$ The following statements are equivalent:
\begin{itemize}
	\item [(a)] $I$ is a completely squarefree lexsegment ideal, that is,  the squarefree shadow of $L(u,v)$ is a  lexsegment set too.
	\item [(b)] For any squarefree monomial $w$ of degree $d,$ $w<_{\lex} v,$ there exists $i> 1$ such that $x_i|w$ and $x_1w/x_i\leq_{\lex}u.$
\end{itemize}
\end{Theorem}

For this class of ideals we have the following result.

\begin{Theorem}[\cite{BS},\cite{BEOS}]\label{linearcomp}
Let $u=x_1x_{i_2}\cdots x_{i_d}$ and $v=x_{j_1}\cdots x_{j_d}$ be two squarefree monomials of degree $d\geq 2$ and $I=(L(u,v))$ the 
squarefree lexsegment ideal generated by the lexsegment set $L(u,v).$ Assume that $I$ is a completely squarefree lexsegment ideal. Let $B$ be 
the set of all the squarefree monomials $w$ of degree $d$ such that $w < _{lex} v$ and $x_1w/x_{\max(w)} > u.$ Then $I$ has a linear 
resolution if and only if $I$ is a final squarefree lexsegment  ideal or the following condition holds:
for all  $(w_1,w_2)\in B\times B $ such that $w_1\neq w_2 $  and $x_1w_1/x_{\min(w_1)}\leq u$, there exists an index $\ell$ such that   $\min(w_1)\leq \ell < \max(w_2),$  $x_\ell| w_2,$ $x_1w_2/x_{\ell}\leq u$ and 
$w_1/x_{\min(w_1)}\neq w_2/x_{\ell}.$
\end{Theorem}

We now consider the particular settings which we are interested in. 

Let $u=x_1x_i$ and $v=x_jx_k$ with $i\geq 3$ and $2\leq j\leq n-2.$ 
According to Theorem \ref{caractcomp} we get the following characterization of the completely lexsegment edge ideals.

\begin{Corollary}
Let $u,v$ be as above and $I=(L(u,v)).$ Then $I$ is a completely lexsegment edge ideal if and only if $j\geq i-2.$
\end{Corollary} 

\begin{proof}
For $w=x_{j+1}x_{j+2}$ we see that $x_1w/x_{j+1}\leq u$ if and only if $j+2\geq i.$
\end{proof}

Next we apply Theorem \ref{linearcomp} and get the following

\begin{Corollary}\label{cor1}
Let $u,v$ be as above and $I=(L(u,v))$ a completely lexsegment edge ideal, that is, $j\geq i-2.$ Then $I$ has a linear resolution if and 
only if $i\leq j+1$ or $i=j+2$ and $v=x_jx_n.$
\end{Corollary}

\begin{proof}
In the case $i\leq j+1$ one may apply Theorem \ref{linearcomp} or simply observe that if we order the minimal monomial 
generators of $I$ as 
\[
x_2x_3, x_2x_4,\ldots,x_2x_n,x_3x_4,\ldots, x_jx_k, x_1x_i, x_1x_{i+1},\ldots,x_1x_n,
\]
then we get linear quotients, hence $I$ has a linear resolution. 

Let $i=j+2.$ If $v=x_jx_n,$ then we have $B=\{x_{j+1}x_{j+2},\ldots,x_{j+1}x_n\}.$ In this case one may choose $\ell=j+1$ in order to verify the condition from Theorem 1.10. Let $v >_{\lex} x_jx_n.$ Then one may choose
$w_1=x_jx_n, w_2=x_{j+1}x_n\in B$. It follows that  $w_1,w_2$ do not satisfy the condition from Theorem \ref{linearcomp} since the only possible
choice for $\ell$ is $\ell=j+1,$ and, in this case, $w_1/x_j=w_2/x_{j+1}.$ 
\end{proof}

Next we consider lexsegment edge ideals which are not complete. To this aim we recall the following

\begin{Theorem}[\cite{BS},\cite{BEOS}]\label{linearnoncomp}
Let $I=L(u,v)$ be a squarefree lexsegment ideal determined by $u=x_1x_{i_2}\cdots x_{i_d}$ and $v=x_{j_1}\cdots x_{j_d}$, $j_1\geq 2.$ 
Assume that $I$ is not a completely squarefree lexsegment ideal. Then $I$ has a linear resolution if and only if $v$ is of the 
form $v=x_{\ell}x_{n-d+2}\cdots x_n$ for some $2\leq \ell < n-d+1.$
\end{Theorem}

Applying the above theorem to our particular setting we get the following

\begin{Corollary}\label{cor2}
Let $I=(L(u,v))$ be a lexsegment edge ideal, where $u=x_1x_i, i\geq 3,$ and $v=x_jx_k,$ $j < i-2$. Then $I$ has a linear resolution
if and only if $v=x_jx_n.$
\end{Corollary}

By using the above results we can compute the regularity of the lexsegment edge ideals.

\begin{Proposition}\label{propreg}
Let $I=(L(u,v))$ be a lexsegment edge ideal where $u=x_1x_i, v=x_jx_k, j\geq 2.$ Then
\[
\reg{I}=\left\{ 
\begin{array}{ll}
	3, & \text{ if } i\geq j+2 \text{ and } x_n\not| v \\
	2, & { otherwise.}
\end{array}\right.
\]
\end{Proposition}

\begin{proof}
The proof follows immediately from Corollaries \ref{cor1} and \ref{cor2}.
\end{proof}

\section{Arithmetical rank of lexsegment edge ideals}
\label{aralexideals}

In this section we aim to prove Theorem \ref{aralei} on the arithmetical rank of lexsegment edge ideals. A useful tool will be 
Schmitt-Vogel Lemma (see \cite{SV}).

\begin{Lemma}\cite{SV}\label{svlemma}
Let $I\subset S$ be a squarefree monomial and $A_1,\ldots, A_r$ be some subsets of the set of monomials of $I.$ Suppose that the 
following conditions hold:
\begin{itemize}
	\item [(SV1)] $|A_1|=1$ and $A_i$ is a  finite set for any $2\leq i\leq r;$
	\item [(SV2)] The union of all the sets $A_i, i=\overline{1,r},$ contains the set of the minimal monomial generators of $I;$
	\item [(SV3)] For any $i\geq 2$ and for any two different monomials $m_1,m_2\in A_i$ there exists $j<i$ and a monomial 
	$m^{\prime}\in A_j$ such that $m^{\prime}|m_1m_2.$
\end{itemize}
Let  $g_i=\sum_{m_i\in A_i}m_i$ for $1\leq i\leq r.$ Then $\sqrt{(g_1,\ldots,g_r)}=I.$ In particular, $\ara(I)\leq r.$  
\end{Lemma}

\begin{Theorem}\label{aralei}
Let $I=(L(u,v))$ be a lexsegment edge ideal. Then 
\[
\ara(I)=\projdim_S(S/I).
\]
\end{Theorem}

\begin{proof} Let $u=x_1x_i$ and $v=x_jx_k$ such that $u\geq_{\lex}v.$
In the first place we observe that the statement is obviously true if $j=1$ since, for instance, $I$ is isomorphic as an $S$-module 
to the ideal generated by the variables $x_i,\ldots,x_k.$  Hence we may assume that $j\geq 2.$ We will consider separately the case 
$j=2$. 

Let $j\geq 3.$ By Corollary \ref{pd1}, we have $\projdim_S(S/I)=n-1$ if and only if $x_{i-1}x_n\geq_{\lex}v.$ If this is the case, then,
by using inequalities (\ref{equstar}), it follows that  
 $n-1=\projdim_S(S/I)\leq \ara(I)\leq n-1,$ and, consequently, the required equality. 
 
 Now let $v>_{\lex}x_{i-1}x_n$ in the same hypothesis on $j,$ namely $j\geq 3.$ We have $\projdim_S(S/I)=n-2.$ We are going to 
 distinguish two cases to study. In both cases we show that $\ara(I)=n-2=\projdim_S(S/I)$ by using Schmitt-Vogel Lemma.
 
 Case (1). Let $i=4$ or $x_{i-1}x_i\geq_{\lex}v>_{\lex}x_{i-1}x_n.$ In particular, by our assumption $j\geq 3,$  we 
 have  $i\geq 4.$ We display the minimal monomial generators of $I$ in an upper triangular tableau as follows. In the first row we put the generators 
 divisible by $x_2$ ordered decreasingly with respect to the lexicographic order except the monomial $x_2x_n$ which is intercalated
between the monomials $x_2x_{i-1}$ and $x_2x_i.$ In the same way we order on the second row the monomials divisible by $x_3,$ 
intercalating the monomial $x_3x_n$ between $x_3x_{i-1}$ and $x_3x_i.$ We continue in this way up to the row containing the monomials divisible by $x_{i-2}.$ On the next row we put the monomials $x_1x_n,x_1x_i,x_1x_{i+1},\ldots,x_1x_{n-1},$ and, finally, on the last row, we put the remaining generators, namely $x_{i-1}x_i,\ldots,v.$ Then our tableau looks as follows.
\[
\begin{array}{lllllllll}
	x_2x_3 & x_2x_4 & \ldots & x_2x_{i-1} & \underline{x_2x_n} & x_2x_i & \ldots & x_2x_{n-2} & x_2x_{n-1}\\
	       & x_3x_4 & \ldots & x_3x_{i-1} & \underline{x_3x_n} & x_3x_i & \ldots & x_3x_{n-2} & x_3x_{n-1}\\   
	       &        &        & \vdots     &     \vdots         &\vdots  &        &   \vdots   & \vdots     \\
	       &        &        & x_{i-2}x_{i-1}&\underline{x_{i-2}x_n} & x_{i-2}x_i & \ldots & x_{i-2}x_{n-2}& x_{i-2}x_{n-1}\\
	       &        &        &               &\underline{x_1x_n} &  \underline{x_1x_i} & \ldots & \underline{x_1x_{n-2}}&
	       \underline{x_1x_{n-1}}\\
	       &        &        &               &                   & x_{i-1}x_i & \ldots \ v
\end{array}
\]

Next we define the sets $A_1,A_2,\ldots,A_{n-2}$ in the following way. In the first set we put the monomial from the left-up corner 
of the tableau. In the second set we put the two monomials from the left up parallel to the diagonal of the triangular tableau. In the third set we collect the three monomials from the next parallel to the diagonal, and so on. Explicitly, the sets are the following ones.
\[
\begin{array}{ll}
	A_1 = & \{x_2x_{n-1}\},\\
	A_2 = & \{x_2x_{n-2}, x_3x_{n-1}\},\\
	A_3 = & \{x_2x_{n-3},x_{3}x_{n-2},x_4x_{n-1}\}, \\
	\vdots &  \\
	A_{n-i+1}= & \{x_2x_n, x_3x_i,x_4x_{i+1},\ldots \},\\
	\vdots &  \\
	A_{n-2}= & \{x_2x_3,x_3x_4,\ldots, x_{i-2}x_{i-1},x_1x_n,x_{i-1}x_i\}.
\end{array}.
\]
One may easy check that the sets $A_1,\ldots,A_{n-2}$ verify all the conditions from Lemma \ref{svlemma}. We give only a brief 
explanation concerning the third condition. Indeed if one picks up two different monomials in the set $A_j$ for some $j\geq 2,$ let us 
say $m_1$ from the $r$-th row and $m_2$ from the $s$-th row of the tableau with $r<s,$ then the monomial $m^{\prime}$ at the 
intersection of the 
$r$-th row and the column of $m_2$ divides the product $m_1m_2$ and $m^{\prime}\in A_{\ell}$ for some $\ell < j.$

Case (2). Let $x_3x_4\geq_{\lex} v=x_jx_k>_{\lex}x_{i-1}x_i$ and $i\geq 5.$ Then we construct a similar triangular tableau to that one 
from the previous case, but we preserve the decreasing lexicographic order in each row. In this tableau we will add the underlined 
monomials in the $(j-1)$-th row.
\[
\begin{array}{llllllllll}
x_2x_3 & x_2x_4 & \ldots & x_2x_j & x_2x_{j+1} & \ldots & x_2x_k & x_2x_{k+1} & \ldots & x_2x_n\\
       &x_3x_4 & \ldots & x_3x_j & x_3x_{j+1} & \ldots & x_3x_k & x_3x_{k+1} & \ldots & x_3x_n\\
       &       &        & \vdots & \vdots     &        & \vdots &  \vdots    &        &  \vdots\\
       &       &        & x_{j-1}x_j & x_{j-1}x_{j+1} & \ldots & x_{j-1}x_k & x_{j-1}x_{k+1} & \ldots & x_{j-1}x_n\\   
       &       &        &            & x_jx_{j+1}& \ldots & x_jx_k=v & \underline{x_1x_jx_{k+1}}& \ldots & \underline{x_1x_jx_n}\\
       &       &        &            &           & x_1x_i &  &  \ldots  &  & x_1x_n     

\end{array}
\]
Note that in this case it is impossible to have $i=j+1$. Indeed, if $i=j+1,$ then, by our hypothesis we have 
$x_jx_k>_{\lex}x_jx_{j+1}$, which is impossible. 

One may easy check that the sets $A_1=\{x_2x_n\}, A_2=\{x_2x_{n-1}, x_3x_n\}, A_3=\{x_2x_{n-2},x_3x_{n-1},x_4x_n\},\ldots, 
A_{n-2}=\{x_2x_3,x_3x_4,\ldots,x_jx_{j+1}\}$ verify the conditions from Lemma \ref{svlemma}, thus $\ara(I)\leq n-2.$ Since we also have 
$\projdim_S(S/I)=n-2,$ we get that $\ara(I)=\projdim_S(S/I).$

To finish the proof, we only need to consider the case $j=2,$ that is, $u=x_1x_i$ and $v=x_2x_k$ for some $i$ and $k$ such that
$v>_{\lex}x_{i-1}x_n.$ Note that, in particular, we have $i-1\geq 2,$ that is, $i\geq 3.$

If $i>k$, then, as in the proof of Lemma \ref{depth3}, we obtain that the Taylor resolution of $I$ is minimal. This implies that $\projdim_S(S/I)=\mu(I)$, where 
$\mu(I)$ denotes the number of the minimal monomial generators of $I.$ Therefore, $\ara(I)=\mu(I)=\projdim_S(S/I).$

If $k\geq i,$ we show that $\ara(I)=\projdim_S(S/I)=n-2$ by using again Lemma \ref{svlemma}. In this case we put the generators of 
$I$ in a $2$-row tableau. 
\[
\begin{array}{llllllll}
	       &        & x_1x_i & \ldots & x_1x_k & x_1x_{k+1} & \ldots & x_1x_n\\
x_2x_3   & \ldots & x_2x_i & \ldots & x_2x_k &            &        &
\end{array}
\]
 If $i> 3,$ we add to the second row the monomials $x_1x_2x_{k+1},\ldots,x_1x_2x_n.$
 
 We get the tableau
\[
\begin{array}{llllllll}
	       &        & x_1x_i & \ldots & x_1x_k & x_1x_{k+1} & \ldots & x_1x_n\\
x_2x_3   & \ldots & x_2x_i & \ldots & x_2x_k &  \underline{x_1x_2x_{k+1}}& \ldots & \underline{x_1x_2x_n}\\
\end{array}
\]
and set
\[
A_1=\{x_1x_2x_n\}, A_2=\{x_1x_n,x_1x_2x_{n-1}\},A_3=\{x_1x_{n-1},x_1x_2x_{n-2}\}, \ldots
\]
\[
\ldots, A_{n-k}=\{x_1x_{k+2},x_1x_2x_{k+1}\},
A_{n-k+1}=\{x_1x_{k+1},x_2x_k\},\ldots, A_{n-2}=\{x_2x_3\}.
\]

If $i=3,$ then we add the monomials $x_1x_2x_{k+1},\ldots, x_1x_2x_{n-1}$ to the initial tableau and get
\[
\begin{array}{llllllll}
x_1x_3	 &    x_1x_4     & \ldots & x_1x_k & x_1x_{k+1} & \ldots & x_1x_{n-1} & x_1x_n\\
x_2x_3   &    x_2x_4     & \ldots  & x_2x_k & \underline{x_1x_2x_{k+1}}& \ldots & \underline{x_1x_2x_{n-1}} & \\
\end{array}
\]
We set 
\[
A_1=\{x_1x_3\}, A_2=\{x_1x_4,x_2x_3\}, A_3=\{x_1x_5,x_2x_4\},\ldots, A_{k-2}=\{x_1x_k,x_2x_{k-1}\},
\]
\[
A_{k-1}=\{x_1x_{k+1},x_2x_k\},A_k=\{x_1x_{k+2},x_1x_2x_{k+1}\},\ldots,A_{n-2}=\{x_1x_n,x_1x_2x_{n-1}\}.
\]
In both cases, by using Lemma \ref{svlemma}, we get $\projdim_S(S/I)=n-2\leq \ara(I)\leq n-2,$ hence $\ara(I)=n-2=\projdim_S(S/I).$
\end{proof}

We recall that an ideal $I\subset S$ is called a \textit{set-theoretic complete intersection} if $\ara(I)=\height(I)$.   For squarefree monomial ideals we  $\ara(I)\geq 
\projdim_S(S/I),$ by using again (\ref{equstar}). If $\height(I)=\ara (I),$ we get 
\[
\height (I)\geq \projdim_S(S/I)=n-\depth (S/I)\geq n-\dim(S/I)=\height (I).
\] 
 Therefore, we derive the following implication for squarefree monomial ideals:
\begin{center}
set-theoretic complete intersection $\Rightarrow$ Cohen-Macaulay.
\end{center}

For  lexsegment edge ideals the converse is also true,  by Theorem \ref{aralei}.

\begin{Corollary}
Let $I$ be a lexsegment edge ideal. Then the following statements are equivalent:
\begin{itemize}
	\item [(a)] $I$ is Cohen-Macaulay.
	\item [(b)] $I$ is a set-theoretic complete intersection.
\end{itemize}
\end{Corollary}

\section{Arithmetical rank of the Alexander dual of a lexsegment edge ideal}
\label{aradual}

As before, let $u=x_1x_i, v=x_jx_k$ be two squarefree monomials of degree $2$ such that $u\geq_{\lex} v$ and $I=(L(u,v))$ the lexsegment edge 
ideal generated by the set $L(u,v).$ 

Let $I^{\ast}$ be the Alexander dual ideal of $I$.
Then we have 
\begin{eqnarray*}
I^{\ast} &= &(x_1,x_{i})\cap(x_1,x_{i+1})\dots \cap (x_1,x_{n})
\cap(x_2,x_3)\cap(x_2,x_{4})\cap \dots \cap (x_2,x_{n})\\
&&\cap(x_3,x_4)\cap(x_3,x_{5})\cap \dots \cap (x_3,x_{n})\cap \dots
\cap(x_j,x_{j+1})\cap \dots \cap (x_j,x_{k}),
\end{eqnarray*}
which is an unmixed ideal of height two (see, e.g., \cite[Proposition 1.1]{T1}).
In this section we show the equality 
$\ara(I^{\ast})=\projdim_S(S/I^{\ast})$. 
Since we have $\projdim_S(S/I^{\ast})=\reg I$ \cite[Corollary 1.6]{T1}, by Proposition \ref{propreg} we have the following:

\begin{Proposition}
Let $I=(L(u,v))$ be a lexsegment edge ideal where $u=x_1x_i, v=x_jx_k, j\geq 2.$ Then
\[
\projdim_S(S/I^{\ast})=\left\{ 
\begin{array}{ll}
	3, & \text{ if } i\geq j+2 \text{ and } x_n\not| v \\
	2, & { otherwise.}
\end{array}\right.
\]
\end{Proposition}

Now we determine the arithmetical rank of the Alexander dual of a lexsegment edge ideal.

\begin{Theorem}\label{aradual}
Let $I=(L(u,v))$ be a lexsegment edge ideal. Then 
\[
\ara(I^{\ast})=\projdim_S(S/I^{\ast}).
\]
\end{Theorem}

\begin{proof}
We may assume that $u=x_1x_i, v=x_jx_k$.
If $j=1$, then $I^{\ast}=(x_1,x_{i})\cap(x_1,x_{i+1}) \cap \dots \cap (x_1,x_{k})
 =(x_1,x_{i}x_{i+1}\dots x_{k})$ is a (set-theoretic) complete intersection.
Hence we may assume that $j \ge 2$.

Now we assume that  $i\leq j+1 \text{ or } k=n$.
Then we have  $\projdim_S(S/I^{\ast})=\height I^{\ast}=2$, and  $S/I^{\ast}$ is Cohen-Macaulay.
In this case $I^{\ast}$ is a set-theoretic complete intersection by Kimura \cite{K}.
Hence $\ara(I^{\ast})=\projdim_S(S/I^{\ast})=2$.

Next we assume that  $i\geq j+2 \text{ and } k\ne n$. 
Let $J^{\ast}$ be the Alexander dual ideal of $J=(L(x_1x_i, x_{j-1}x_n))$.
Then we have $\ara(J^{\ast})=\projdim_S(S/J^{\ast})=2$.
Hence there exist $f_1,f_2 \in S$ such that
$\sqrt{(f_1,f_2)}=J^{\ast}.$
Then we have
\begin{eqnarray*}
I^{\ast} &= &J^{\ast} \cap(x_j,x_{j+1})\cap(x_j,x_{j+2})\cap \dots \cap (x_j,x_{k})\\
&= &\sqrt{(f_1,f_2)} \cap(x_j,x_{j+1}x_{j+2}\dots x_{k})\\
&= & \sqrt{(f_1f_2)(x_j,x_{j+1}x_{j+2}\dots x_k)}\\
&= &\sqrt{(x_jf_1, x_jf_2, x_{j+1}x_{j+2}\dots x_{k}f_1, x_{j+1}x_{j+2}\dots x_{k}f_2)}\\
&= &\sqrt{(x_jf_1, x_jf_2+x_{j+1}x_{j+2}\dots x_{k}f_1, x_{j+1}x_{j+2}\dots x_{k}f_2)}.
\end{eqnarray*}
For the last equality we need only to justify the inclusion from the left part to the right part. This follows immediately if we notice 
that $x_jf_2$ and $x_{j+1}x_{j+2}\dots x_k f_1$ are solutions of the equation 
\[
t^2-(x_jf_2 + x_{j+1}x_{j+2}\dots x_k f_1)t+ x_jx_{j+1}x_{j+2}\dots x_k f_1f_2=0.
\]
We have $3 =\projdim_S(S/I^{\ast}) \le \ara(I^{\ast}) \le 3$.
Hence $\ara(I^{\ast})=\projdim_S(S/I^{\ast})=3$, as desired.
\end{proof}

\end{document}